\documentclass[12pt]{article}
\usepackage[a4paper, total={5.5in, 8.5in}]{geometry}

\usepackage{amsmath,amsthm,amssymb}
\usepackage[T1]{fontenc} 
\usepackage[utf8]{inputenc} 
\usepackage{graphicx}
\usepackage{scrextend}

\graphicspath{ {images/} }
\usepackage{hyperref} 

\newtheorem{theorem}{Theorem}[section]

\theoremstyle{definition}

\newtheorem{proposition}[theorem]{Proposition}

\newcommand{\LabelQuote}[2]{\vspace{0.5cm}%
     \parbox{12.4cm}{\em #1}\hspace*{0.5cm}(#2)\\[0.5cm]}
     
\newcommand{\NextEq}{\refstepcounter{equation}\theequation}

\begin{document}
\title{Domination number of Token Graphs
\author{Ruy Fabila-Monroy \thanks{Departamento de Matem\'aticas, CINVESTAV.} 
\footnote{\tt{ruyfabila@math.cinvestav.edu.mx}}  
\and Sergio Gerardo G\'omez-Galicia\footnotemark[1] \footnote{\texttt{sgomez@math.cinvestav.mx}} }}

\maketitle
\begin{abstract}
The $k$-token graph of $G$ is the graph, $F_k(G)$, whose vertices are all the $k$-subsets of $V(G)$; with two of them adjacent whenever their symmetric difference is a pair of adjacent vertices in $G$. In this paper, we study the domination number of the token graphs of the  star, $S_n$, and the 
complete graph, $K_n$.
\end{abstract}

\section{Introduction}

Let $G=(V,E)$ be a simple graph on $n$ vertices, and let $1 \le k \le n-1 $ be an integer. The \emph{$k$-token graph}, $F_k(G)$, is the graph whose vertex set is the set of all $k$-subsets of vertices of $G$, being two of them adjacent if their symmetric difference is a pair of adjacent vertices of $G$. This graph was introduced by Fabila-Monroy, Flores-Peñaloza, Huemer, Hurtado, Urrutia, and Wood \cite{Token}. The authors gave the following interpretation.  Assume that there are $k$ indistinguishable tokens placed on the vertices of $G$, with one token per vertex at most. Construct a new graph with vertex set all the possible token configurations, and make two of them adjacent if one token configuration can be reached from another by sliding a token along an edge of $G$ to an unoccupied vertex. It is easy to see that this graph is isomorphic to $F_k(G)$. 

Token graphs have been defined independently several times and under different names such as \emph{k-graph of a graph},  \emph{double vertex graph}, \emph{k-tuple vertex graph}, \emph{symmetric k-th power graph}
(see~\cite{double_vertex,audenaert,Token,Johns,rudolph,ktuple}). In Figure \ref{fig:F2S5_F3S5}, the star graph $S_5$ is shown together with its $2$ and $3$-token graphs.

\begin{figure}
	\centering	\includegraphics[width=0.6\linewidth]{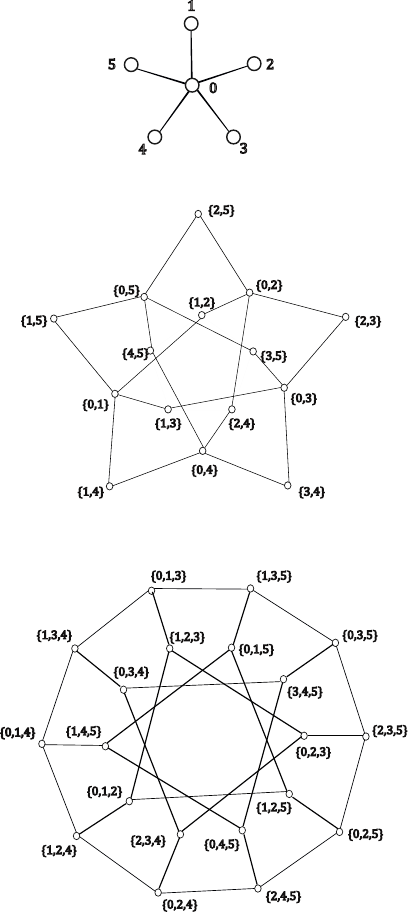}
	\caption{The star graph $S_5$ on the top; and the $2$ and $3$-token graphs of the star graph $S_5$.}   
	\label{fig:F2S5_F3S5}
\end{figure}

A \emph{dominating set} of $G$ is a subset $D\subseteq V(G)$, such that for every vertex of $G$, we have that it is in $D$ or has a neighbor in $D$. The \emph{domination number} of $G$, denoted $\gamma(G)$, is the number of vertices in a dominating set of $G$ of minimum cardinality, i.e. 
\[\gamma(G):=\{|D|:D\text{ is a dominating set of $G$}\}.\]
Let $X\subseteq V(G)$ and $v \in V(G)$. If $v \in D$ or $v$ has a neighbor in $X$, then we say that $X$ \emph{dominates} $v$.

The concept of dominance in graph theory has been of great relevance in different fields, such as network security, social network modeling, and combinatorial optimization (see \cite{mohanty,werner, henning2022graph}). See \cite{dayap2025domination} for a further discussion of applications of domination in graph theory.

The \emph{Johnson graph}, denoted $J(n,k)$, is defined as the graph with vertex set given by the set of all $k$-subsets of an $n$-set, where two such vertices are adjacent whenever they intersect in exactly $k-1$ elements. From this definition, it's straightforward to see that $k$-token graph of the complete graph, $F_k({K_n})$ is isomorphic to the Johnson graph $J(n,k)$.

In this paper, we study the domination number of the token graphs of the star and complete graph. We first give some notation and previous results in the following section.

\section{Preliminaries}\label{section_prel}

 Note that, in general, $F_{k}(G) \cong F_{n-k}(G)$. Therefore, we always assume $1 \leq k \leq \frac{n}{2}$. Let $S$ be a finite set. The set of all $k$-sets of $S$ is denoted by $\binom{S}{k}$. We use the notation $[n]:=\{1,\dots,n\}$; $[m,n]:=[n]-[m-1]$; and $[n]_i:=\{j\in [n]:i\mid j\}$. Let $\Delta(G)$ and $\delta(G)$ be the maximum and minimum degrees of vertices of $G$, respectively. 

 The star graph on $n$ leaves is denoted by $S_{n}$. We assume that its vertex set is $V(S_n):=\{0\} \cup[n]$, with $0$ as the center vertex of the star. We assume that the vertex set of the complete graph is given by $V(K_n)=[n]$. 
 
 The following well-known result gives us a general lower bound for $\gamma(G)$ relating the maximum degree.

\begin{proposition}\label{prop_gen_lwb}
Let $G$ be a graph on $n$ vertices. Then,
    \[
    \frac{n}{1+\Delta(G)}\leq \gamma(G).
    \]
\end{proposition}
\qed

Now we state a theorem that gives an upper bound in terms of the minimum degree. A proof of the theorem can be found in Spencer and Alon \cite{alon2016probabilistic}. 

\begin{theorem}
    Let $G$ be a graph on $n$ vertices, such that $\delta(G)>1$. Then
    \[
    \gamma(G)\leq n\left(\frac{1+\ln{(\delta(G)+1)}}{\delta(G)+1}\right).
    \]
\end{theorem}

An \emph{independent set} of $G$ is a subset of $I\subset V(G)$, non two of which are adjacent. The \emph{independence number} of $G$, denoted $\alpha(G)$, is the number of vertices in a maximum independent set. The following well-known result gives us an upper bound of $\gamma(G)$ relating the independence number.

\begin{proposition}\label{prop_upbound_indep}
    \[
    \gamma(G)\leq \alpha(G).
    \]
\end{proposition}

The independence number of token graphs was studied by Abdelmalek, Meulen, E. Kevin, V Meulen and Van Tuyl in \cite{wellcovered}. The authors give the following result.

\begin{theorem}[\cite{wellcovered}]\label{teo_indep}
\[
\binom{\alpha(G)}{k}\leq \alpha(F_k(G))\leq \frac{1}{k}\binom{n}{k-1}\alpha(G).
\]
\end{theorem}

We say that a $k$-set $K \subset [n]$, \emph{covers} an $l$-set of $L \subset [n]$, if $L\subset K$. 
Let $2\leq l<k<n$ be integers. An $(n,k,l)$ \emph{covering design} is a family $\mathcal{F}$ of $k$-sets of $[n]$ that
cover every $l$-set of $[n]$. The \emph{covering number}, denoted $C(n,k,l)$, denotes the minimum size of the family $\mathcal{F}$. Since every $k$-set covers $\binom{k}{l}$ $l$-sets, we have the trivial bound
\[C(n,k,l) \ge \frac{\binom{n}{l}}{\binom{k}{l}}.\]
The following theorem, due to Rödl \cite{rodl1985packing} gives an upper bound asymptotically close to this lower bound. 
\begin{theorem}[\cite{rodl1985packing}]\label{theo_covers}
    Let $2\leq l<k<n$ be fixed integers. Then,
    \[
    C(n,k,l)\leq (1+o(1))\frac{\binom{n}{l}}{\binom{k}{l}}.
    \]
\end{theorem}

The $(n,3,2)$ covering designs such that every $2$-set of $[n]$ is covered exactly by one $3$-set of $[n]$ are called \emph{Steiner triple systems}.
Note that a Steiner triple system has exactly
\[\frac{1}{3}\binom{n}{2}=\binom{n(n-1)}{6} \]
elements. This implies that a necessary condition for the existence of a Steiner triple system is that $n\equiv 1,3\pmod 6$.
Skolem~\cite{skolem} and Bose~\cite{bose} independently showed that this condition is also sufficient.
\begin{theorem}[\cite{skolem,bose}]\label{teo_steiner}
 \[C(n,3,2) =\frac{1}{3}\binom{n}{2} \]
 if and only if $n\equiv 1,3\pmod 6$.
\end{theorem}

\section{Star graph}\label{section_result}

\begin{theorem}\label{prop_dom_2_star}
\[
\gamma(F_2(S_n))=n-1.
\]
\end{theorem}

\begin{proof}

Let \[V_0:=\{A\in V(F_2(S_n)):0\in A\}; \textrm{ and } W:=V(F_2(S_n))\setminus V_0;\]
for every $1 \le i  \le n$, let 
\[V_i:=\{A\in V(F_2(S_n))\setminus V_0:i\in A\};\]
and for every $1 \le i <j  \le n$, let 
\[D_{ij}:= (V_0\setminus\{\{0,i\},\{0,j\}\})\cup \{\{i,j\}\}.\]
 
We show that $D_{ij}$ is a dominating set. Let $A\in V_0$. If $A\notin D_{ij}$, then $A=\{0,i\}$ or $A=\{0,j\}$, and, in any case, $A$ is adjacent to $\{i,j\}\in D_{ij}$. Now let $A\notin V_0$. Thus, $A\in V_t$ for some $t$. If $t\neq i,j$, then $A$ is adjacent to $\{0,t\}\in D_{ij}$. Now, if $t=i$ or $t=j$, then, $A=\{i,j\}\in D_{ij}$ or $|A\cap\{i,j\}|=1$. 
Suppose that  $|A\cap\{i,j\}|=1$ and let $l =A \setminus \{i,j\}$. Thus, $A$ is adjacent to $\{0,l\}\in D_{ij}$, and $D_{ij}$ is a dominating set.
Since 
\[|D_{ij}|=(n-2)+1=n-1,\]
we have that
 \[\gamma(F_2(S_n)) \le n-1.\]

Let $X\subset V(F_2(S_n))$ with $|X|=n-2$ elements. We show that
there is at least one vertex of $W$, that is not in $X$ nor adjacent to a vertex of $X$. 
Let $X_0:=V_0 \cap X$ and $x:=|X_0|$. Let $X_1 =X\setminus X_0$; thus, $|X_1|=n-2-x$.
Every vertex $\{0,i\} \in V_0$ has exactly $n-1$ neighbors in $W$(and no neighbors in $V_0$); these
are of the form $\{i,j\}$ for some $j\neq i$. Moreover, a vertex $\{i,j\} \in W$ has exactly
two neighbors, and they are both in $V_0$; these are $\{0,i\}$ and $\{0,j\}$. Therefore, the number of neighbors of $X_0$ in
$W$ is exactly 
\[x(n-1)-\binom{x}{2}.\]
Thus, the number of vertices of $W$ that are either in $X$ ($X_1$) or adjacent
to a vertex of $X$ ($X_0$) is given by the function
\[f(x)=x(n-1)-\binom{x}{2}+n-2-x.\]
Since \[f'(x)=n-x-\frac{3}{2} \ge \frac{1}{2} >0\]
for all $x \in [0,n-2]$, we have that $f$ attains its maximum at $x=n-2$.
But, 
\[f(n-2)=\frac{n(n-1)}{2}-1<\binom{n}{2}=|W|.\]
Therefore, \[\gamma(F_2(S_n)) \ge n-1.\]
This completes the proof.
\end{proof}
%
%

%

\begin{theorem}\label{thm:S_n}
Let $1 \le k \le n-1$ be fixed, and let $p$ be the smallest prime that divides $k-1$. Then
\[\frac{1}{k}\binom{n}{k-1} < \gamma(F_k(S_n)) \le \left ( 1-\frac{1}{p}+  \frac{1}{k+p-1} +o(1) \right )\binom{n}{k-1}.\]
\end{theorem}
\begin{proof} 
Since $\Delta(F_k(S_n))=n-k+1$, by Proposition~\ref{prop_gen_lwb} we have that
\begin{align*}
 \gamma(F_k(S_n)) & \ge \frac{1}{1+(n-k+1)}\cdot \binom{n+1}{k}\\[0.6em]
 &= \frac{1}{n-k+2}\cdot \frac{n+1}{k}\binom{n}{k-1}\\[0.6em]
 &=\frac{n+1}{k(n-k+2)}\binom{n}{k-1}\\[0.6em]
 & > \frac{1}{k} \binom{n}{k-1}.
\end{align*}

We now prove the upper bound. Partition the vertices of $F_k(S_n)$ into the  sets
 \[V_0  := \{A \in V(F_k(S_n)): 0 \in A\} \textrm{ and } \overline{V_0}  := \{A \in V(F_k(S_n)): 0 \notin A\}.\]
 Let 
\[D_1:=\left \{A \in V_0: \sum_{a \in A} a \not \equiv 1 \pmod{p}\right \}.\]

Let $A \in \overline{V_0}$ and $S:=\sum_{a \in A} a$. If there exists $a \in A$ such that
 $S-a \not \equiv  1\pmod{p}$, then  there exists an element of $D_1$ adjacent to $A$.
 Suppose that  $S-a \equiv  1\pmod{p}$ for all $a \in A$.
Summing over all $a \in A$,
 we have that,
 \begin{align*}
0
&\equiv (k-1)S \\[0.6em]
&\equiv kS - \sum_{a\in A} a \\[0.6em]
&\equiv \sum_{a\in A}(S-a) \\[0.6em]
&\equiv \sum_{a\in A} 1 \\[0.6em]
&\equiv k \\[0.6em]
&\equiv 1 \pmod p.
\end{align*}
We arrive at a contradiction; therefore,

\LabelQuote{every vertex of $\overline{V_0}$ is adjacent to a vertex of $D_1$.}{\NextEq}

We now bound the cardinality of $D_1$ using a probabilistic argument. Let $A$ be chosen uniformly at random
from $V_0$. We bound the probability that $A \in D_1$. 
 We can think of $A$ being chosen by the following random process, first we choose $0$ to be
in $A$. For $i=1,\dots,k$, let $A_i$ be the elements of $A$ chosen so far, so that $A_1=\{0\}$ and $A_k=A$. For $i=1,\dots,k-1$, we choose $a_{i+1}$ uniformly at random
from $[n]\setminus A_i$. Let $S:=\sum_{a \in A_{k-1}}A_{k-1}$, note that $A$ is in $D_1$, if and only if, $a_k \not \equiv 1-S \pmod{p}$. There at least
$\lfloor n/p \rfloor-(k-2) \ge n/p-(k-1)$ and at most $\lceil n/p\rceil \le n/p+1$ such elements in $[n] \setminus A_{k-1}$.
Therefore, \[ \frac{1}{p}-\frac{p-1}{p} \cdot\frac{k-1}{n-(k-1)} \le\operatorname{Pr} \left ( A \in D_0\right ) \le \frac{1}{p}+\frac{k+p-1}{n-k+1}. \]
This implies that
\begin{equation}
 |D_1| = \left ( \frac{1}{p} \pm o(1)\right) \binom{n}{k-1} \textrm{ and } |V_0\setminus D_1| = \left ( \frac{1}{p}\pm o(1)\right) \binom{n}{k-1}.
 \end{equation}
The only vertices that remain to be dominated are those in 
\[V_0 \setminus D_1=\left \{A \in V_0: \sum_{a \in A} a \equiv 1 \pmod{p}\right \}.\]

For a given set $X \subset [n]$ and $0 \le i \le p-1$, let 
\[X_i:=\left \{x \in X: x\equiv i \pmod{p} \right \}.\]
For every $A \in V_0$, let $A'=A \setminus \{0\}$, so $A'$ is a $k-1$ set.
Let $A \in V_0 \setminus D_1$. Let $r(A)$ be the largest residue class modulo $p$ of the elements in $A'$, where ties are broken 
by taking the smallest residue. Formally,
\[r(A):=\min \left \{i: \left|  A_i'\right |  \ge \left| A_j' \right |  \textrm{ for all } 0 \le j \le p-1 \right \}.\]
Note that 
\[\left |A_{r(A)}' \right |\ge \frac{k-1}{p}.\]
For every $0 \le i \le p-1$ and $0 \le j \le p-1$, let 
\[V_{ij}':=\left \{A \in V_0\setminus D_1: r(A)=i \textrm{ and } \left | A_{r(A)}'\right |=j \right \}.\]
The $V_{ij}'$ form a partition of $V_0 \setminus D_1$.
Let \[U_{ij}:=\left \{A_{i}':A \in V_{ij} \right \} \textrm{ and } W_{ij}:=\left \{ A'\setminus A_{i}':A \in V_{ij}\right \}.\]
We can write $V_{ij}'$ as a ``product'' of $U_{ij}$ and $W_{ij}$, in the sense that
\[V_{ij}'=\{0\}\cup \{X \cup Y: X \in U_{ij} \textrm{ and } Y \in W_{ij} \}.\]
Therefore,
\[\left|V_{ij}' \right |=\left|U_{ij} \right | \cdot \left|W_{ij} \right |.\]
Note that 
\[U_{ij}=\binom{[n]_i}{j}.\]
By Theorem~\ref{theo_covers}, there exists a subset \[U_{ij}' \subset \binom{[n]_i}{j+1}\]
that covers every element of $U_{ij}$ and such that
\[\left | U_{ij}'\right | \le (1+o(1))\frac{1}{j+1}\left |U_{ij} \right | \le (1+o(1)) \frac{p}{k+p-1}\left |U_{ij} \right |\]
Let \[ D_{ij}'=\{X \cup Y: X \in U_{ij}' \textrm{ and } Y \in W_{ij} \}.\]
We have that

\LabelQuote{every vertex of $V_{ij}'$ is adjacent to a vertex of $D_{ij}'$.}{\NextEq}

\noindent Moreover, 
\[|D_{ij}'|\le (1+o(1)) \frac{p}{k+p-1} \left |U_{ij} \right | \cdot \left |W_{ij} \right |= (1+o(1)) \frac{p}{k+p-1} |V_{ij}'|.\]

Let \[D_2:=\bigcup_{\substack{0 \le i \le p-1 \\[0.3em] 0 \le j \le p-1}} D_{ij}'.\]
Thus, 
\[D=D_1 \cup D_2\]
is a dominating set of $F_k(S_n)$; and
\[|D| \le \left ( 1-\frac{1}{p}+  \frac{1}{k+p-1} \right )\binom{n}{k-1}.\]
\end{proof}

Using the fact that every set of integers whose sum is odd must have an odd number of odd integers,
we can improve the upper bound of Theorem~\ref{thm:S_n} when $k \equiv 1 \pmod{4}$ to
\[\gamma(F_k(S_n)) \le \left (\frac{1}{2}+  \frac{1}{k+3} +o(1) \right )\binom{n}{k-1}.\]

\section{Complete graph}

\begin{proposition}\label{prop_dom_2_complete}
\[
\gamma(F_2(K_n))=\left\lfloor\frac{n}{2}\right\rfloor.
\]
\end{proposition}

\begin{proof}
Let $D=\{A_1,\dots,A_{\lfloor n/2\rfloor}\} \subset V(F_2(K_n))$, such that $A_i\cap A_j=\emptyset$ for every $i\neq j$. Let $B\in V(F_2(K_n))$. Assume that $B\notin D$. Thus, there is some $A_i$ such that $|A_i\cap B|=1$. Therefore, $B$ is adjacent to $A_i$. Hence, $D$ is a dominating set.
Now we show that there is no smaller dominating set. 
Let $X\subset V(F_2(K_n))$, with $|X|<\lfloor n/2\rfloor$. Note that \[\left |V(K_n) \setminus \bigcup_{A \in X} A \right | \ge 2.\]
Thus, there exists $B \in V(F_2(K_n)) \setminus X$. We have $A \cap B =\emptyset$ for all $A \in X$.
Therefore, $X$ is not a dominating set of $F_2(K_n)$.  This completes the proof. 
\end{proof} 

\bigskip

\begin{theorem} \ \\
\begin{itemize}
\item If $n$ is even, then 
\[\frac{1}{12}n^{2}-\frac{1}{6}n\leq \gamma(F_3(K_n))\leq (1+o(1))\left(\frac{1}{12}n^{2}-\frac{1}{6}n\right);\]   

\item if $n$ is odd, then 
\[\frac{1}{12}n^{2}-\frac{1}{6}n+\frac{1}{4}\leq \gamma(F_3(K_n))\leq (1+o(1))\left(\frac{1}{12}n^{2}-\frac{1}{6}n+\frac{1}{4}\right);\]   
and

\item if  $n\equiv 2,6\pmod{12}$, then 
\[\gamma(F_3(K_n))=\left(\frac{1}{12}n^{2}-\frac{1}{6}n\right).\]
\end{itemize}
\end{theorem}
\begin{proof}  

Let $D$ be a dominating set of $F_3(K_n)$.
Let $E$ be the $2$-sets of $[n]$ not covered by an element of $D$. Let $H$ be the graph with vertex set equal to $[n]$ and
edge set equal to $E$. Note that $H$ is triangle free since a triangle $\{i,j,k\}$ of $H$ would be a vertex of $F_3(K_n)$ not in $D$
nor adjacent to a vertex of $D$ (since all its $2$-sets are in $E$). By Mantel's theorem we have that 
\[E \le \left \lfloor \frac{n^2}{4} \right \rfloor.\]

Since every vertex of $D$ covers exactly three $2$-sets of $[n]$. We have that
\[\binom{n}{2} \le 3|D|+|E| \le  3|D|+\left \lfloor \frac{n^2}{4} \right \rfloor.\]
This implies that 
\[|D| \ge \frac{1}{3}\left (\binom{n}{2} - \left \lfloor \frac{n^2}{4} \right \rfloor \right ).\]
Note that
\[\left \lfloor \frac{n^2}{4} \right \rfloor=
\begin{cases}
\frac{n^2}{4}  & \textrm{ if } n \textrm{ is  even;}\\[0.6em]
\frac{n^2-1}{4}  & \textrm{ if } n \textrm{ is  odd;}
\end{cases}
\]
Therefore,
\[|D| \ge 
\begin{cases} 
\frac{n^2}{12}-\frac{n}{6}  & \textrm{ if } n \textrm{ is  even;}\\[0.6em]
\frac{n^2}{12}-\frac{n}{6} +\frac{1}{4} & \textrm{ if } n \textrm{ is  odd.}
\end{cases}\]

We now show the upper bound. Let \[V_1:=\{1,2,\dots,\lfloor n/2\rfloor\} \textrm{ and } V_2:=\{\lfloor n/2\rfloor+1, \lfloor n/2\rfloor+2,\dots, n\}.\]
By Theorem~\ref{theo_covers}, there exists subsets $D_1 \subset \binom{V_1}{3}$ and  $D_2 \subset \binom{V_2}{3}$ such that:
 every $2$-set of $V_1$ is covered by an element of $D_1$; every $2$-set of $V_2$ is covered by an element of $D_2$; and
  \[|D_1| \le (1+o(1))\frac{1}{3}\binom{\frac{\lfloor n \rfloor}{2}}{2} \textrm{ and } |D_2|\le (1+o(1))\frac{1}{3}\binom{\frac{\lceil n \rceil}{2}}{2}.\] 
By Theorem~\ref{teo_steiner}, if $n\equiv 2,6\pmod{12}$ in addition we can choose $D_1$ and $D_2$ so that 
 \[|D_1| = \frac{1}{3}\binom{\frac{\lfloor n \rfloor}{2}}{2} \textrm{ and } |D_2| = \frac{1}{3}\binom{\frac{\lceil n \rceil}{2}}{2}.\]
 
Let $A \in F_3(K_n)$. Since $\{V_1,V_2\}$ is a partition of $[n]$, we have that either
\[\left |A \cap V_1 \right | \ge 2 \textrm{ or } \left |A \cap V_2 \right | \ge 2.\]
Suppose  that $ \left |A \cap V_1 \right | \ge 2$. 
There exists is an element $B \in D_1$  that covers a $2$-set of $A \cap D_1$.
Thus, $A \in D_1$ or $A$ is adjacent to an element of $D_1$.
Similarly, if  $\left |A \cap V_2 \right | \ge 2$, we have that $A \in D_2$ or $A$ is adjacent to an element of $D_2$.
Therefore,
\[D:=D_1 \cup D_2\]
is a dominating set $F_2(K_n)$.
Simple arithmetic shows that 
\[\binom{\frac{\lfloor n \rfloor}{2}}{2}+\binom{\frac{\lceil n \rceil}{2}}{2} =
\begin{cases}
 \frac{n^2}{4}-\frac{n}{2} & \textrm{ if } n \textrm{ is  even; }\\[0.6em] \frac{n^2}{4}-\frac{n}{2}+\frac{1}{4}  & \textrm{ if } n \textrm{ is  odd. }
\end{cases}
\]
The result follows.
\end{proof}

\begin{theorem}
 \[\frac{1}{k^2}\binom{n}{k-1} <  \gamma(F_k(K_n)) \le \frac{1}{k}\binom{n}{k-1}.\]
\end{theorem}
\begin{proof}
 Since $\alpha(K_n)=1$, the upper bound follows directly from Proposition~\ref{prop_upbound_indep} and Theorem~\ref{teo_indep}.
 Note that every vertex of $F_k(K_n)$ has degree equal to $k(n-k)$. By Propostion~\ref{prop_gen_lwb},
 we have that
 \begin{align*}
  \gamma(F_k(K_n))  & \ge \frac{1}{1+k(n-k)} \binom{n}{k}\\
   &=\frac{1}{1+k(n-k)} \cdot \frac{n-k+1}{k}\binom{n-1}{k}\\
   &=\frac{k(n-k+1)}{nk-k^2+1} \cdot \frac{1}{k^2}\binom{n-1}{k}\\
   &=\frac{nk-k^2+k}{nk-k^2+1} \cdot \frac{1}{k^2} \binom{n-1}{k}\\
   &>\frac{1}{k^2} \binom{n-1}{k}.\\
 \end{align*}
\end{proof}

\section*{Acknowledments}

This work was initiated during the \emph{Escuela de Verano 2025, Departamento de Matemáticas, CINVESTAV.} We thank
the participants for various helpful discussions.


\small \bibliographystyle{alpha} 
\bibliography{Domination}

\newcommand{\etalchar}[1]{$^{#1}$}
\begin{thebibliography}{FMFPH{\etalchar{+}}12}

\bibitem[ABEL91]{double_vertex}
Yousef Alavi, Mehdi Behzad, Paul Erd{\H{o}}s, and Don~R Lick.
\newblock Double vertex graphs.
\newblock {\em J. Combin. Inform. System Sci}, 16(1):37--50, 1991.

\bibitem[AGRR07]{audenaert}
Koenraad Audenaert, Chris Godsil, Gordon Royle, and Terry Rudolph.
\newblock Symmetric squares of graphs.
\newblock {\em Journal of Combinatorial Theory, Series B}, 97(1):74--90, 2007.

\bibitem[AMMVT20]{wellcovered}
Fred~M Abdelmalek, Esther~Vander Meulen, Kevin N~Vander Meulen, and Adam Van~Tuyl.
\newblock Well-covered token graphs.
\newblock {\em arXiv preprint arXiv:2010.04539}, 2020.

\bibitem[AS16]{alon2016probabilistic}
Noga Alon and Joel~H Spencer.
\newblock {\em The probabilistic method}.
\newblock John Wiley \& Sons, 2016.

\bibitem[Bos39]{bose}
Raj~Chandra Bose.
\newblock On the construction of balanced incomplete block designs.
\newblock {\em Annals of Eugenics}, 9(4):353--399, 1939.

\bibitem[DCA{\etalchar{+}}25]{dayap2025domination}
Jonecis~A Dayap, Leomarich~F Casinillo, Bijo~S Anand, Joey~S Estorosos, and Ricky~B Villeta.
\newblock Domination in graph theory: A bibliometric analysis of research trends, collaboration and citation networks.
\newblock {\em arXiv preprint arXiv:2503.08690}, 2025.

\bibitem[FMFPH{\etalchar{+}}12]{Token}
Ruy Fabila-Monroy, David Flores-Pe{\~n}aloza, Clemens Huemer, Ferran Hurtado, Jorge Urrutia, and David~R Wood.
\newblock Token graphs.
\newblock {\em Graphs and Combinatorics}, 28(3):365--380, 2012.

\bibitem[HVV22]{henning2022graph}
Michael~A Henning and Jan~H Van~Vuuren.
\newblock {\em Graph and Network Theory: An Applied Approach Using Mathematica{\textregistered}}, volume 193.
\newblock Springer Nature, 2022.

\bibitem[Joh88]{Johns}
Garry~L Johns.
\newblock {\em Generalized distance in graphs}.
\newblock PhD thesis, Western Michigan University, 1988.

\bibitem[MMRD16]{mohanty}
Jasaswi~Prasad Mohanty, Chittaranjan Mandal, Chris Reade, and Ariyam Das.
\newblock Construction of minimum connected dominating set in wireless sensor networks using pseudo dominating set.
\newblock {\em Ad Hoc Networks}, 42:61--73, 2016.

\bibitem[R{\"o}d85]{rodl1985packing}
Vojt{\v{e}}ch R{\"o}dl.
\newblock On a packing and covering problem.
\newblock {\em European Journal of Combinatorics}, 6(1):69--78, 1985.

\bibitem[Rud02]{rudolph}
Terry Rudolph.
\newblock Constructing physically intuitive graph invariants.
\newblock {\em arXiv preprint arXiv:quant-ph/0206068}, 2002.

\bibitem[Sko58]{skolem}
Th~Skolem.
\newblock Some remarks on the triple systems of steiner.
\newblock {\em Mathematica Scandinavica}, pages 273--280, 1958.

\bibitem[Wer20]{werner}
Frank Werner.
\newblock Graph-theoretic problems and their new applications, 2020.

\bibitem[ZLLA92]{ktuple}
Bi~Wen Zhu, Jiuqiang Liu, Don~R. Lick, and Yousef Alavi.
\newblock {$n$}-tuple vertex graphs.
\newblock In {\em Proceedings of the {T}wenty-third {S}outheastern {I}nternational {C}onference on {C}ombinatorics, {G}raph {T}heory, and {C}omputing ({B}oca {R}aton, {FL}, 1992)}, volume~89, pages 97--106, 1992.

\end{thebibliography}

 \end{document}